\documentclass[11pt, oneside]{article}
\usepackage{amsfonts,amssymb,amsmath,latexsym,amsthm}
\usepackage{mathrsfs}
\usepackage{color}
\usepackage[colorlinks]{hyperref}
\usepackage{enumerate,indentfirst,mathtools,verbatim,authblk}
\usepackage{tikz}
\usepackage{psfrag}
\usepackage{graphicx}
\usepackage{bm}
\usepackage[rflt]{floatflt}
\usepackage{float}
\usepackage[square, comma, sort&compress, numbers]{natbib}
\numberwithin{equation}{section}
\allowdisplaybreaks

\newcommand{\upcite}[1]{\textsuperscript{\textsuperscript{\cite{#1}}}}

\newtheorem{theorem}{Theorem}[section]

\newtheorem{lemma}[theorem]{Lemma}
\newtheorem{corollary}[theorem]{Corollary}
\theoremstyle{definition}
\newtheorem{definition}[theorem]{Definition}

\theoremstyle{remark}
\newtheorem{remark}[theorem]{Remark}
\numberwithin{equation}{section}

\newenvironment{proof1}{\medskip\noindent{\bf Proof of Theorem \ref{thm1}:}\enspace}{\hfill \qed  \medskip}
\newenvironment{proof2}{\medskip\noindent{\bf Proof of Theorem \ref{thm2}:}\enspace}{\hfill \qed  \medskip}
\newenvironment{proof3}{\medskip\noindent{\bf Proof of Theorem  \ref{thm3}:}\enspace}{\hfill \qed  \medskip}
\newenvironment{proof4}{\medskip\noindent{\bf Proof of Theorem \ref{thm4}:}\enspace}{\hfill \qed  \medskip}
\newenvironment{proofaddthm2}{\medskip\noindent{\bf Proof of Theorem \ref{addthm2}:}\enspace}{\hfill \qed  \medskip}

\usepackage{amssymb,amsmath}
\begin{document}
\title{\LARGE\bf {Global existence and blow-up of solutions to a class of  non-Newton filtration equations with singular potential and logarithmic nonlinearity}}
\author{Menglan Liao, Zhong Tan\thanks{Corresponding author: Zhong Tan
\newline\hspace*{6mm}{\it Email
address:}~liaoml14@mails.jlu.edu.cn(M. Liao),~ztan85@xmu.edu.cn(Z. Tan)}}
\affil{School of Mathematical Sciences, Xiamen University, Xiamen, Fujian, 361005, China.}
\renewcommand*{\Affilfont}{\small\it}
\date{} \maketitle
\vspace{-20pt}

{\bf Abstract:}
In this paper, a class of  non-Newton filtration equations with singular potential and logarithmic nonlinearity under initial-boundary condition is investigated. Based on potential well method and Hardy-Sobolev inequality, the global existence of solutions is derived when the initial energy $J(u_0)$ is subcritical($J(u_0)<d$), critical($J(u_0)=d$) with $d$ being the mountain-pass level. Finite time blow-up results are obtained as well when the initial energy $J(u_0)$ satisfies specific conditions. Moreover, the upper and lower bounds of the blow-up time are given.



\thispagestyle{empty}
\section{Introduction}
In this paper, we are concerned with the following initial-boundary problem:
\begin{equation}
\label{1.1}
\begin{cases}
      |x|^{-s}u_{t}-\mathrm{div}(|\nabla u|^{p-2}\nabla u)=|u|^{q-2}u\ln|u|& \text{in}~\Omega \times(0,T),  \\
      u(x,t)=0&\text{on}~\partial\Omega \times(0,T),  \\
      u(x,0)=u_0(x)& \text{for}~x\in\Omega,
\end{cases}
\end{equation}
where the initial value $u_0(x)\in W_0^{1,p}(\Omega)$, $T\in (0,\infty]$ is the maximal existence time of solutions, $\Omega\subset \mathbb{R}^N(N>p)$ is a bounded domain containing the origin $0$ with smooth boundary $\partial\Omega$, $x=(x_1,x_2,\cdots,x_N)\in \mathbb{R}^N$ with $|x|=\sqrt{x_1^2+x_1^2+\cdots+x_N^2}$, and the parameters satisfy
\begin{equation}
\label{1.2}
\quad p\ge 2,\quad 0\le s\le 2,\quad p<q<\frac{Np}{N-p}.
\end{equation}

The volumetric moisture content $\theta(x)$, the macroscopic velocity $\overrightarrow{V}$ and the density of the fluid $u$, under the assumption that  a compressible fluid flows in a homogeneous isotropic rigid porous medium, are governed by the following equation \cite{WZYL2001}:
\[\theta(x)u_t-\mathrm{div}(u\overrightarrow{V})=f(u),\]
where $f(u)$ is the source.  For the non-Newtonian fluid, provided that the fluid investigated is the polytropic gas, one obtains
\begin{equation}
\label{1.3}
\theta(x)u_t-c^\alpha\lambda\mathrm{div}(|\nabla u^m|^{p-2}\nabla u^m)=f(u),
\end{equation}
where $c>0,~m>0,~\lambda>0,~p\ge 2.$ Over past years, many researchers has paid attention to  equation \eqref{1.3}.  For the source $f(u)=u^q$, much work has been obtained. For instance, for $\theta(x)=|x|^{-2}$ and $m=1$, in 2004, Tan  \cite{T2004} considered the existence and asymptotic estimates of global solutions as well as finite time blow-up of local solutions based on the classical Hardy inequality \cite{HLP1934}. Later on, Wang \cite{W2007} extended the results obtained by Tan to $\theta(x)=|x|^{-s}$ with $0\le s\le 2$,  and  proved the existence of global solutions by Hardy-Sobolev inequality \cite{BT2002},  and found two sufficient conditions for blow-up in finite time by the combination of variational methods and classical concavity methods.   For $\theta(x)=|x|^{-s}$ with $0\le s\le 2$ in equation \eqref{1.3},  Zhou \cite{Z2016} discussed the global existence and finite time blow-up of solutions  by potential well method and Hardy-Sobolev inequality when the initial energy $J(u_0)$ is subcritical, i.e. $J(u_0)<d$ with $d$ being the mountain-pass level.  For $J(u_0)\ge d$, Xu and Zhou \cite{XZ2018} discussed the behaviors of the solution by using the potential well method and some differential inequality techniques. Their results, in fact extended previous ones obtained by Hao and Zhou \cite{HZ2017}, where some blow-up conditions with $J(u_0)\ge d$ were obtained for $m=1$, $p=2$ and $\theta(x)=|x|^{-s}$ with $0\le s\le 2$ in equation \eqref{1.3}.   When the source $f(u)$ is logarithmic nonlinearity,  Deng and Zhou \cite{DZ2020}  investigated the following semilinear heat equation
with singular potential and logarithmic nonlinearity
\[|x|^{-s}u_{t}-\Delta u=u\ln|u|\]
under an appropriate initial-boundary value condition. They did make full use of the logarithmic Sobolev inequality in \cite{G1975,DD2002} to handle the difficulty caused by the logarithmic nonlinear term $u\ln|u|$. Taking the combination of a family of potential wells, the existence of global solutions and infinite time blow-up solutions were obtained.

Inspired by the results mentioned above, it is pretty natural to discuss what will happen if one replaces $f (u)$ in equation \eqref{1.3} by $|u|^{q-2}u\ln|u|$, and makes $m=1$ and $\theta(x)=|x|^{-s}$ with $0\le s\le 2$, i.e. problem \eqref{1.1}?  Noticing that if one wishes to use the logarithmic Sobolev inequality, the diffusion term $\mathrm{div}(|\nabla u|^{p-2}\nabla u)$ and the logarithmic nonlinearity $|u|^{q-2}u\ln|u|$ must have the same growth oder, i.e. $p=q$. In this paper, we obviously cannot use the logarithmic Sobolev inequality due to $p<q$. In oder to handle the logarithmic nonlinearity, we will use some pivotal lemmas proposed in \cite{DZ2019}. To deal with the difficulty caused by $|x|^{-s}$, we have to apply Hardy-Sobolev inequality. The main purpose of this paper is to discuss the global existence and finite time blow-up of solutions to problem \eqref{1.1}. Firstly, by  potential well method which was introduced by Payne and Sattinger  \cite{PS1975}, and by combining Hardy-Sobolev inequality, the global existence of solutions is derived when the initial energy $J(u_0)\le d$.  In particular,  the asymptotic behavior of solutions is presented. Secondly, some blow-up results are shown. To be more precise, we prove that the solution blows up in finite time supposed that one of the following three assumptions holds: 
\begin{enumerate}[$(1)$]
  \item the initial energy $J(u_0)<0$;
  \item the initial energy $J(u_0)\le M\le d$ and $I(u_0)<0$, $M$ is defined in \eqref{20g};
  \item  the initial energy $0<J(u_0)<\frac{1}{C_2}\Big\||x|^{-\frac s2}u_0(x)\Big\|_2^2-\frac{C_1}{C_2}$, with  $C_1=\frac{\widetilde{C}}{2}$ and  $C_2=\frac {pq}{q-p}\frac{\widetilde{C}}{2}$, here $\widetilde{C}$ is shown in \eqref{20h}. 
\end{enumerate}
And we give the upper bound of the blow-up time under each assumption. Finally, by using the interpolation inequality and Sobolev embedding theorem, the lower bound of the blow-up time is derived as well under the assumption (3).

The rest of this paper is organized as follows. In Section 2, the main results of this paper are stated. In Section 3, we give some preliminaries including notations and lemmas. Proofs of  the main results are given in Section 4.

\section{Main results}
Throughout this paper, $C$ is the given constant of Hardy-Sobolev inequality in Lemma \ref{lem2.5}. We denote by $\|\cdot\|_p$ and $\|\nabla (\cdot)\|_p$ the norm on $L^p(\Omega)$ and $W_0^{1,p}(\Omega)$, respectively. And denote by $(\cdot,\cdot)$ the inner product in $L^2(\Omega)$.  In order to present our main results, let us begin with introducing some notations, definitions.

 For any $u\in W_0^{1,p}(\Omega)$, define the energy functional $J$ and the Nehari functional $I$ as follows:
\begin{equation}
\label{19add3}
J(u)=\frac 1p\|\nabla u\|_p^p-\frac 1q\int_\Omega |u|^q\ln|u|dx+\frac {1}{q^2}\|u\|_q^q,
\end{equation}
\begin{equation}
\label{19add4}
I(u)=\|\nabla u\|_p^p-\int_\Omega |u|^q\ln|u|dx.
\end{equation}
Clearly,  the functionals $J$ and $I$ are well-defined and continuous on $W_0^{1,p}(\Omega)$, and 
\begin{equation}
\label{19add5}
J(u)=\frac 1qI(u)+\frac {q-p}{pq}\|\nabla u\|_p^p+\frac {1}{q^2}\|u\|_q^q.
\end{equation}

Define 
\begin{equation}
\label{19add6}
d=\inf_{u\in \mathcal{N}}J(u),
\end{equation}
where $\mathcal{N}:=\{u\in W_0^{1,p}(\Omega)\backslash\{0\}|I(u)=0\}$ is the Nehari manifold. In Lemma \ref{2lem20}, we obtain
\begin{equation}
\label{20g}
d\ge M:=\frac{q-p}{pq}r_*^p.
\end{equation}
And define the potential well $\mathcal{W}$ and its corresponding set $\mathcal{V}$ by
\[\mathcal{W}:=\{u\in W_0^{1,p}(\Omega)|I(u)>0,~J(u)<d\}\cup \{0\},\]
\begin{equation}
\label{20a}
\mathcal{V}:=\{u\in W_0^{1,p}(\Omega)|I(u)<0,~J(u)<d\}.
\end{equation}

\begin{definition}\label{def2.1}
$($Weak solution$)$ A  function $u:=u(x,t)\in L^\infty(0,T;W_0^{1,p}(\Omega))$  with $|x|^{-\frac s2}u_t\in L^2(0,T;L^2(\Omega))$ is called a weak solution of problem \eqref{1.1} on $\Omega\times [0,T)$ if $u(x,0)=u_0(x)$ in $W_0^{1,p}(\Omega)$ and 
\begin{equation}
\label{19add1}
(|x|^{-s}u_t,\phi)+(|\nabla u|^{p-2}\nabla u,\nabla\phi)=(|u|^{q-2}u\ln|u|,\phi)\quad\text{ a.e. }t\in (0,T)
\end{equation}
 for any $\phi\in W_0^{1,p}(\Omega)$. Moreover, 
 \begin{equation}
\label{19add2}
\int_0^t\Big\||x|^{-\frac s2}u_\tau\Big\|_2^2d\tau+J(u(x,t))=J(u_0)\quad\text{ a.e. }t\in (0,T).
\end{equation}
\end{definition}

\begin{definition}\label{def2.2}
$($Finite time blow-up$)$ Let $u$ be a weak solution of problem \eqref{1.1}. We say that $u$ blows up at some finite time $T$ if  
\[\lim_{t\to T^-}\Big\||x|^{-\frac s2}u\Big\|_2^2=+\infty.\]
\end{definition}

Global existence of solutions is presented as follows: 
\begin{theorem}\label{thm1}
Let \eqref{1.2} hold and $u_0(x)\in W_0^{1,p}(\Omega)$. Assume that $J(u_0)<d$ and $I(u_0)>0$ , then  problem \eqref{1.1} admits a global solution $u\in L^\infty(0,\infty;W_0^{1,p}(\Omega))$  with $|x|^{-\frac s2}u_t\in L^2(0,\infty;L^2(\Omega))$, and $u(t)\in \mathcal{W}$ for $0 \le t < \infty$. Moreover, if $J(u_0)<d(\alpha):=\frac{q-p}{pq}(r(\alpha))^p\le d$, then 
\begin{equation*}
\Big\||x|^{-\frac s2}u\Big\|_2^2\le
\begin{cases}
 \Big\||x|^{-\frac s2}u_0\Big\|_2^2e^{-\frac{2}{\widetilde{C}}\Big[1-\Big(\frac{2q}{q-2}\frac{1}{(r(\alpha))^2}J(u_0)\Big)^{\frac{q+\alpha-2}{2}}\Big]t}&\text{for }p=2,\\
\Big\{\Big(\frac p2-1\Big)\frac{2}{\widetilde{C}^{\frac p2}}\Big[1-\Big(\frac{pq}{q-p}\frac{1}{(r(\alpha))^p}J(u_0)\Big)^{\frac{q+\alpha-p}{p}}\Big]t+\Big(\Big\||x|^{-\frac s2}u_0\Big\|_2^2\Big)^{1-\frac p2}\Big\}^{\frac{2}{2-p}}&\text{for }p>2,
 \end{cases}
\end{equation*}
where $\alpha$ satisfies \eqref{20e},  $r(\alpha)$ and $\widetilde{C}$ are defined in \eqref{20f} and \eqref{20h}, respectively.
\end{theorem}

\begin{remark}
For any $u\in\mathcal{N}$, it follows from \eqref{19add5} and Lemma $\ref{21add}$ that 
\[J(u)=\frac 1qI(u)+\frac {q-p}{pq}\|\nabla u\|_p^p+\frac {1}{q^2}\|u\|_q^q\ge \frac {q-p}{pq}\|\nabla u\|_p^p>\frac {q-p}{pq}(r(\alpha))^p.\]
The definition of $d$ indicates $d\ge d(\alpha).$
\end{remark}

\begin{corollary}\label{cor1}
Let \eqref{1.2} hold and $u_0(x)\in W_0^{1,p}(\Omega)$. Assume that $J(u_0)=d$ and $I(u_0)\ge 0$ , then  problem \eqref{1.1} admits a global solution $u\in L^\infty(0,\infty;W_0^{1,p}(\Omega))$  with $|x|^{-\frac s2}u_t\in L^2(0,\infty;L^2(\Omega))$ and $u(t)\in \mathcal{W}\cup \partial\mathcal{W}$ for $0 \le t < \infty$.
\end{corollary}
Based on the proof of Theorem \ref{thm1}, following the proof of Corollary 1: Case 3 in \cite{DZ2019}, one directly can prove Corollary \ref{cor1}. In the present paper, we omit the process.

In what follows, we introduce the finite time blow-up results. For the sake of simplicity, throughout this paper, we set 
  \begin{equation}
\label{5.1}
L(t):=\frac 12\Big\||x|^{-\frac s2}u\Big\|_2^2.
\end{equation}
\begin{theorem}\label{thm2}
Let \eqref{1.2} hold. If $J (u_0)< 0$, and $u$ is a weak solution to problem \eqref{1.1}, then $u$ blows up at some finite time $T$ with
 \[T\le  \frac{2L(0)}{(2-q)qJ(u_0)}.\]
\end{theorem}

\begin{theorem}\label{addthm2}
Let \eqref{1.2} hold. If $J (u_0)\le M$ defined in \eqref{20g}, $I(u_0)<0$, and $u$ is a weak solution to problem \eqref{1.1}, then $u$ blows up at some finite time $T$. In addition, 
 \[T\le \frac{8(q-1)L(t_0)}{(q-2)^2q(M-J(u(x,t_0)))}+t_0.\]
 where $t_0\ge 0$ is any finite time  such that $J(u(x,t_0))<M$. In particular, for $J (u_0)< M$, the finite blow-up time satisfies 
\[T\le \frac{8(q-1)L(0)}{(q-2)^2q(M-J(u_0))}.\]
\end{theorem}

Next, we give the finite time blow-up result and estimate the upper
and lower bounds of the blow-up time for the arbitrarily high initial energy.
\begin{theorem}\label{thm3}
Let $u$ be a weak solution of problem \eqref{1.1}, and \eqref{1.2} hold. If 
\begin{equation}
\label{20add4}
0<C_2J(u_0)<L(0) -C_1
\end{equation}
with  $C_1=\frac{\widetilde{C}}{2}$ and  $C_2=\frac {pq}{q-p}\frac{\widetilde{C}}{2}$, here $\widetilde{C}$ is shown in \eqref{20h}, 
then $u$ blows up at some finite time $T$. Moreover, 
the upper of the blow-up time is given by 
\[T\le \frac{4(q-1)p\widetilde{C}L(0)}{(q-2)^2(q-p)F(0)}.\]
with $F(0)=L(0) -C_1-C_2E(u_0)>0.$
\end{theorem}

\begin{theorem}\label{thm4}
Suppose that all conditions of  Theorem $\ref{thm3}$ are fulfilled and $q+\alpha<p\big(1+\frac 2N\big)$. Then the lower bound of the blow-up time can be estimate by
 \begin{equation*}
\label{20add6}
T\ge \frac{L^{1-\kappa}(0)}{\alpha^{-1}\Big[C_*^{\theta(q+\alpha)}\alpha^{-\frac{\theta(q+\alpha)}{p}}[\text{diam}(\Omega)]^{\frac{s(1-\theta)(q+\alpha)}{2}}\Big]^{\frac{p}{p-\theta(q+\alpha)}}2^\kappa(\kappa-1)},
\end{equation*}
here $\alpha$ is  defined in \eqref{20e}, $\theta=\Big(\frac{1}{2}-\frac{1}{q+\alpha}\Big)\Big(\frac{1}{2}-\frac{N-p}{Np}\Big)^{-1}\in(0,1)$, $\kappa=[\frac{(1-\theta)(q+\alpha)}{2}]/[1-\frac{\theta(q+\alpha)}{p}]$ and $C_*$ is the optimal constant of embedding $W_0^{1,p}(\Omega)\hookrightarrow L^{\frac{Np}{N-p}}(\Omega)$.
\end{theorem}

\section{Preliminaries}
In this section, we give some lemmas, which are of great significance in the proofs of our main results. Since the definitions of $J$ and $I$ are same as \cite{DZ2019}, here we directly borrow some lemmas without detailed proofs. 

Firstly, we give some notations obtained in \cite{DZ2019}. Let \eqref{1.2} hold. For any $\alpha$ satisfying 
\begin{equation}
\label{20e}
0<\alpha\le \frac{Np}{N-p}-q,
\end{equation}
define 
\begin{equation}
\label{20f}
r(\alpha):=\Big(\frac{\alpha}{B_\alpha^{q+\alpha}}\Big)^{\frac{1}{q+\alpha-p}},
\end{equation}
 where $B_\alpha$ is the optimal embedding constant of $W_0^{1,p}(\Omega)\hookrightarrow  L^{p+\alpha}(\Omega).$

 \begin{lemma}\label{21add}
\upcite{DZ2019}Let \eqref{1.2} hold and $u\in W_0^{1,p}(\Omega)\backslash\{0\}$. Then for any $\alpha$ satisfying \eqref{20e}, we have
\begin{enumerate}[$(1)$]
  \item if $0<\|\nabla u\|_p\le r(\alpha)$, then $I(u)>0$;
  \item if $I(u)\le 0$, then $\|\nabla u\|_p> r(\alpha)$,
\end{enumerate}
where $r(\alpha)$ is defined in \eqref{20f}.
 \end{lemma}
 
\begin{lemma} \label{lem20}
\upcite{DZ2019}Let \eqref{1.2} hold. Then 
\[r_*:=\sup_{\{\alpha \text{ satisfies }\eqref{20e}\}}r(\alpha)\]
exists and $0<r_*\le r^*<\infty$, where 
\[r^*:=\sup_{\{\alpha \text{ satisfies }\eqref{20e}\}}\sigma(\alpha) \text{ with } \sigma(\alpha):=\Big(\frac{\alpha}{\kappa^{q+\alpha}}\Big)^{\frac{1}{q+\alpha-p}}|\Omega|^{\frac{\alpha}{q(q+\alpha-p)}}.\]
Here, $|\Omega|$ is the measure of $\Omega$, $\kappa$ is the optimal embedding constant of $W_0^{1,p}(\Omega)\hookrightarrow  L^{q}(\Omega).$
\end{lemma}

\begin{lemma}\label{1lem20}
\upcite{DZ2019}Let \eqref{1.2} hold and $u\in W_0^{1,p}(\Omega)\backslash\{0\}$.
\begin{enumerate}[$(1)$]
  \item If $0<\|\nabla u\|_p<r_*$, then $I(u)>0$;
  \item If $I(u)\le 0$, then $\|\nabla u\|_p\ge r_*$,
\end{enumerate}
where $r_*$ is defined in Lemma $\ref{lem20}$.
\end{lemma}

\begin{lemma}\label{2lem20}
\upcite{DZ2019}Let \eqref{1.2} hold. Then we have
\[d\ge \frac{q-p}{pq}r_*^p,\]
where $d$ is defined in \eqref{19add6} and $r_*$ is defined in Lemma $\ref{lem20}$.
\end{lemma}

\begin{lemma}\label{20lemadd1}
Let $u$ be a weak solution to problem \eqref{1.1}. Then for all $t \in [0, T)$, 
\begin{equation}
\label{20add1}
\frac{d}{dt}\Big\||x|^{-\frac s2}u\Big\|_2^2=-2I(u).
\end{equation}
\end{lemma}
The proof of  Lemma \ref{20lemadd1} directly follows by choosing $\phi=u$ in Definition \ref{def2.1}.
 
  \begin{lemma}\label{20lemadd2}
Let \eqref{1.2} hold and $u_0(x)\in W_0^{1,p}(\Omega)$. Assume that $u$ is a weak solution of problem \eqref{1.1} in $\Omega\times[0,T) $.
\begin{enumerate}[$(1)$]
  \item If $J(u_0)<d$ and $I(u_0)>0$, then $u(t)\in \mathcal{W}$ for $0\le t<T$.
  \item If $J(u_0)<d$ and $I(u_0)<0$, then $u(t)\in \mathcal{V}$ for $0\le t<T$.
  \item If $J(u_0)=d$ and $I(u_0)<0$, then there exists $0<t_0<T$ such that $u(t)\in \mathcal{V}$ for $t_0\le t<T$.
\end{enumerate}
\end{lemma}
\begin{proof}
\begin{enumerate}[$(1)$]
  \item For $J(u_0)<d$, $I(u_0)>0$, from the definition of $\mathcal{W}$, we know $u_0\in\mathcal{W}$. Next we will prove $u(t)\in \mathcal{W}$ for $0<t<T$. In fact, if it is false, there exists a $t_0 \in (0,T)$ such that $u(t_0)\in \mathcal{\partial W}$, which  implies that $u(t_0)\in W_0^{1,p}(\Omega)\backslash\{0\}$ and $J(u(t_0))=d$ or $I(u(t_0))$=0. From \eqref{19add2}, $J(u(t_0))=d$ is not true. So $u(t_0)\in \mathcal{N}$, then by the definition of $d$, we have $J(u(t_0)) \ge d$, which also contradicts \eqref{19add2}. Hence, $u(t) \in \mathcal{W}$ for $0<t<T$.    
 \item By the definition of $\mathcal{V}$,  we have $u_0\in \mathcal{V}$. Next we will show that $u(t)\in \mathcal{V}$ for $0< t<T$. If not, there exist
a $t^0\in(0, T)$ such that $u(t^0)\in\partial \mathcal{V}$, namely
$I(u(t^0))=0$ or $J(u(t^0))=d.$ 
By \eqref{19add2}, we can see that $J(u(t^0))<d$, then $I(u(t^0))=0$. We assume that $t^0$ is the first time such that
$I(u(t^0))=0$, then $I(u(t))<0$ for $0\le t<t^0$. 
Recalling Lemma \ref{1lem20}, one has $
\|\nabla u\|_p\ge  r_*\text{ for }0\le  t<t^0.$ 
Thus, 
\begin{align*}
\|\nabla u(t^0)\|_p=\lim\limits_{t\rightarrow t^{0}}\|\nabla u\|_p\geq r_*>0,
\end{align*} 
which together with $I(u(t^0))=0$ implies that $u(t^0)\in \mathcal{N}$.
By the definition of $d$, we again obtain $J(u(t^0))\geq d$, a contradiction to \eqref{19add2}. 
\item Firstly, we claim that $I(u(x,t))<0$ for $0\le t<\infty$ if $J(u_0)=d$ and $I(u_0)<0$. If the claim is not true, then by the continuity of $I(u(x,t))$,  we assume that $t_0$ is the first time such that
$I(u(t_0))=0$, then $I(u(t))<0$ for $0\le  t<t_0$.  Recalling Lemma \ref{1lem20}, one has $
\|\nabla u\|_p\ge  r_*\text{ for }0\le  t<t_0.$ 
Thus, 
\begin{align*}
\|\nabla u(t_0)\|_p=\lim\limits_{t\rightarrow t_{0}}\|\nabla u\|_p\geq r_*>0,
\end{align*} 
which together with $I(u(t_0))=0$ implies that $u(t_0)\in \mathcal{N}$. Thus, by the definition of $d$, we get 
\begin{equation}
\label{20b}
J(u(x,t_0))\ge d.
\end{equation} 
On the other hand,  it follows from \eqref{20add1} that 
\begin{equation*}
\frac{d}{dt}\Big\||x|^{-\frac s2}u\Big\|_2^2=-2I(u(t))>0\quad \text{for }0\leq t<t_0.
\end{equation*}
Therefore, \eqref{19add2} yields 
\begin{equation}
\label{20c}
J(u(x,t))\le J(u(x,t_0))=J(u_0)-\int_0^{t_0}\Big\||x|^{-\frac s2}u_t\Big\|_2^2dt<d\quad \text{for }0\leq t<t_0,
\end{equation}
which contradicts \eqref{20b}. Taking $t_0$ as the initial time, and following the proof of case (2), we can prove that $u(t)\in \mathcal{V}$ for $t_0\le t<T$.
\end{enumerate}
\end{proof}

 \begin{lemma}\label{addlem2.4}
\upcite{LL2017}Let $\mu$ be a positive number. Then we have the following inequalities:
\[s^{p}\ln s\le\frac{e^{-1}}{\mu}s^{p+\mu}\quad\text{for all }s\ge 1,\]
and
\[\Big|s^{p}\ln s\Big|\le (ep)^{-1}\quad\text{for all }0<s<1.\]
\end{lemma}
\begin{lemma}\label{lem2.4}
\upcite{L1973,LG2017}Suppose a positive, twice-differentiable function $\psi (t)$ satisfies the inequality
\begin{equation*}
\psi''(t)\psi(t)-(1+\theta)({\psi'}(t))^2\ge 0\quad\text{for }t\ge t_0\ge 0,
\end{equation*}
where $\theta>0.$ If $\psi(t_0)>0,~\psi'(t_0)>0$, then $\psi(t)\to\infty$ as $t\to t_1\le t_2=\frac{\psi(t_0)}{\theta\psi'(t_0)}+t_0$.
\end{lemma}
\begin{lemma}\label{lem2.5}
\upcite{BT2002}$($Hardy-Sobolev inequality$)$ Let $\mathbb{R}^N=\mathbb{R}^k\times\mathbb{R}^{N-k}$, $2\le k\le N$ and $x=(y,z)\in \mathbb{R}^N=\mathbb{R}^k\times\mathbb{R}^{N-k}$. For given $n,~\beta$ satisfying $1<n<N$, $0\le \beta\le n$ and $\beta<k$, let $\gamma(\beta,N,n)=n(N-\beta)/(N-n)$. Then there exists a positive constant $C$ depending on $\beta,~n,~N$ and $k$ such that for any $u\in W_0^{1,n}(\mathbb{R}^N)$, it holds
\[\int_{\mathbb{R}^N}\frac{|u(x)|^\gamma}{|y|^\beta}dx\le C\Big(\int_{\mathbb{R}^N}|\nabla u|^ndx\Big)^{\frac{N-\beta}{N-n}}.\]
\end{lemma}

 \section{Proofs of the main results}
 \begin{proof1} The proof will be divided into two steps.
 
 \textsc{Step 1. Global existence} 
 Let $\{\phi_j(x)\}$ be a system of basis in $W_0^{1,p}(\Omega)$ which is orthogonal in $L^2(\Omega)$ and construct the approximate solution $u^m(x,t)$ to problem \eqref{1.1}
 \[u^m(x,t)=\sum_{j=1}^{m}a_j^m(t)\phi_j(x)\quad \text{for }m=1,2,\cdots,\]
which satisfy for $j=1,2,\cdots,m$
\begin{equation}
\label{4.1}
(|x|^{-s}u^m_t,\phi_j)+(|\nabla u^m|^{p-2}\nabla u^m,\nabla\phi_j)=(|u^m|^{q-2}u^m\ln|u^m|,\phi_j),
\end{equation}
\begin{equation}
\label{4.2}
u^m(x,0)=\sum_{j=1}^{m}b_j^m\phi_j(x)\to u_0(x)\text{ in }W_0^{1,p}(\Omega).
\end{equation}
The standard theory of ordinary differential equations yields that there exists a $T>0$ depending on $b_j^m$ for $j=1,2,\cdots,m$ such that $a_j^m(t)\in C^1([0,T])$ and $a_j^m(0)=b_j^m$. As a consequence, $u^m\in C^1([0,T],W_0^{1,p}(\Omega)).$

It follows by multiplying \eqref{4.1} by  $(a_j^m(t))'$, summing for $j$ from $1$ to $m$, and integrating from $0$ to $t$  that 
\[\int_0^t\Big\||x|^{-\frac s2}u_\tau^m\Big\|_2^2d\tau+J(u^m(x,t))=J(u^m(x,0))\quad\text{for }0\le t<T.\]
Noticing that $u^m(x,0)\to u_0(x)$ in $W_0^{1,p}(\Omega)$, we get
\[J(u^m(x,0))\to J(u_0(x))<d\quad\text{and }I(u^m(x,0))\to I(u_0(x))>0.\]
Therefore, for sufficiently large $m$ and for any $0\le t<T$, one has
\begin{equation}
\label{4.3}
\int_0^t\Big\||x|^{-\frac s2}u_\tau^m\Big\|_2^2d\tau+J(u^m(x,t))=J(u^m(x,0))<d\quad\text{and }I(u^m(x,0))>0.
\end{equation}

Now, let us prove that $u^m(x,t)\in \mathcal{W}$ for sufficiently large $m$ and for $0\le t<T$. Otherwise, there exists a $t_0\in (0,T)$ such that $u^m(x,t_0)\in \partial\mathcal{W}$. Recalling  that $0$ is an interior point of $\mathcal{W}$, it follows that 
\[(1)~I(u^m(x,t_0))=0,~\|\nabla u^m(x,t_0)\|_p\neq 0\quad \text{or }(2)~J(u^m(x,t_0))=d.\]
It follows from \eqref{4.3} that $J(u^m(x,t_0))<d$, which implies that (2) cannot happen. If (1) happens, then by the definition of $d$, we get $J(u^m(x,t_0))\ge d$, which contradicts \eqref{4.3}.

Since $u^m(x,t)\in \mathcal{W}$ for sufficiently large $m$ and for $0\le t<T$, one has $I(u^m(x,t))\ge 0$. Combining 
\[J(u^m)=\frac1qI(u^m)+\frac{q-p}{pq}\|\nabla u^m\|_p^p+\frac{1}{q^2}\|u^m\|_q^q\]
and \eqref{4.3} is to get
\[\int_0^t\Big\||x|^{-\frac s2}u^m_\tau\Big\|_2^2d\tau+\frac{q-p}{pq}\|\nabla u^m\|_p^p+\frac{1}{q^2}\|u^m\|_q^q< d\]
for sufficiently large $m$ and for $0\le t<T$, which indicates  that 
\begin{equation}
\label{4.4}
\|u^m\|_{W_0^{1,p}(\Omega)}^p< \frac{dpq}{q-p},
\end{equation}
\begin{equation}
\label{4.5}
\int_0^t\Big\||x|^{-\frac s2}u^m_\tau\Big\|_2^2d\tau< d,
\end{equation}
\begin{equation}
\label{4.6}
\int_\Omega\Big||\nabla u^m|^{p-2}\nabla u^m\Big|^{\frac{p}{p-1}}dx=\|\nabla u^m\|_p^p< \frac{dpq}{q-p}.
\end{equation}
Applying Lemma  \ref{addlem2.4} yields 
\begin{equation}
\label{4.7}
\begin{split}
&\int_\Omega \Big||u^m|^{q-2}u^m\ln|u^m|\Big|^{\frac{q}{q-1}}dx\\
&\quad=\int_{\{x\in\Omega;|u^m|\ge1\}} \Big||u^m|^{q-2}u^m\ln|u^m|\Big|^{\frac{q}{q-1}}dx+\int_{\{x\in\Omega;|u^m|<1\}} \Big||u^m|^{q-2}u^m\ln|u^m|\Big|^{\frac{q}{q-1}}dx\\
&\quad\le (e\mu)^{-\frac{q}{q-1}}\int_{\{x\in\Omega;|u^m|\ge1\}} |u^m|^{(q-1+\mu)\frac{q}{q-1}}dx+(e(q-1))^{-\frac{q}{q-1}}|\Omega|\\
&\quad\le (e\mu)^{-\frac{q}{p-1}}(C^*)^{(q-1+\mu)\frac{q}{q-1}}\|u^m\|_{W_0^{1,p}(\Omega)}^{(q-1+\mu)\frac{q}{q-1}}+(e(q-1))^{-\frac{q}{q-1}}|\Omega|\\
&\quad< (e\mu)^{-\frac{q}{p-1}}(C^*)^{(q-1+\mu)\frac{q}{q-1}}\Big(\frac{dpq}{q-p}\Big)^{(q-1+\mu)\frac{q}{p(q-1)}}+(e(q-1))^{-\frac{q}{q-1}}|\Omega|,
\end{split}
\end{equation}
here we choose $\mu$ such that $(q-1+\mu)\frac{q}{q-1}<\frac{Np}{N-p}$, and $C^*$ is the optimal embedding constant of $W_0^{1,p}(\Omega)\hookrightarrow L^{(q-1+\mu)\frac{q}{q-1}}(\Omega)$.

As a consequence, by \eqref{4.4}-\eqref{4.7},  $T=\infty$.  And for any $\widetilde{T}>0$ we get a subsequence of 
$\{u^m\}$ (still denoted by $\{u^m\}$) such that as $m\to \infty$,
\begin{equation}
\label{4.8}
\begin{cases}
u^m\rightharpoonup u\text{ weakly * in }L^\infty(0,\widetilde{T};W_0^{1,p}(\Omega));\\
|x|^{-\frac s2}u^m_t \rightharpoonup |x|^{-\frac s2}u_t \text{ weakly  in }L^2(0,\widetilde{T};L^2(\Omega));\\
|\nabla u^m|^{p-2}\nabla u^m\rightharpoonup \xi\text{ weakly * in }L^\infty(0,\widetilde{T};L^{\frac{p}{p-1}}(\Omega));\\
|u^m|^{q-2}u^m\ln|u^m|\rightharpoonup \eta\text{ weakly * in }L^\infty(0,\widetilde{T};L^{\frac{q}{q-1}}(\Omega)).
\end{cases}
\end{equation}
Since $\Omega$ is a bounded domain, then
\[\int_0^t\|u^m_\tau\|_2^2d\tau\le (diam(\Omega))^{s}\int_0^t\Big\||x|^{-\frac s2}u^m_\tau\Big\|_2^2d\tau\le (diam(\Omega))^{s}d,\]
which indicates 
\begin{equation}
\label{4.9}
u_t^m\rightharpoonup u_t \text{ weakly in }L^2(0,\widetilde{T};L^2(\Omega)).
\end{equation}
Combining \eqref{4.9} and the first one in \eqref{4.8}, noticing that $W_0^{1,p}(\Omega)\stackrel{compact}{\hookrightarrow} L^q(\Omega)\hookrightarrow L^2(\Omega)$, and then using the Aubin-Lions-Simon compactness lemma \cite{A1963,L1969,S1987}, one has
\begin{equation}
\label{4.10}
u^m\to u \text{ strongly in }C([0,\widetilde{T}];L^q(\Omega)).
\end{equation}
Thus, $u^m\to u$ a.e. $x\in \Omega$, and then $\eta=|u|^{q-2}u\ln|u|$. Recalling \eqref{4.2}, one has $u(x,0)=u_0(x)\in W_0^{1,p}(\Omega).$

We are now in a position to prove that $u$ is a weak solution to problem \eqref{1.1} for any $\widetilde{T}>0$. Let us fix an integer $k$ and choose a function $\omega\in C^1([0,\widetilde{T}],W_0^{1,p}(\Omega))$ with the following 
\[\omega=\sum_{j=1}^kl_j(t)\phi_j(x),\]
where $\{l_j(t)\}_{j=1}^k$ are arbitrarily  given $C^1$ functions. Taking $m \ge k$ in \eqref{4.1}, multiplying \eqref{4.1} by $l_j(t)$,
summing for $j$ from $1$ to $k$, and integrating with respect to $t$ from 0 to $\widetilde{T}$, we obtain
\begin{equation}
\label{4.11}
\int_0^{\widetilde{T}}(|x|^{-s}u^m_t,\omega)dt+\int_0^{\widetilde{T}}(|\nabla u^m|^{p-2}\nabla u^m,\nabla \omega)dt=\int_0^{\widetilde{T}}(|u^m|^{q-2}u^m\ln|u^m|,\omega)dt,
\end{equation}
Particularly, one has
\begin{equation*}
\int_0^{\widetilde{T}}(|x|^{-s}u^m_t,u^m)dt+\int_0^{\widetilde{T}}(|\nabla u^m|^{p-2}\nabla u^m,\nabla u^m)dt=\int_0^{\widetilde{T}}(|u^m|^{q-2}u^m\ln|u^m|,u^m)dt,
\end{equation*}
that is 
\begin{equation}
\label{4.12}
\frac12\int_\Omega|x|^{-s}(u^m(x,\widetilde{T}))^2dx-\frac12\int_\Omega|x|^{-s}(u^m(x,0))^2dx+\int_0^{\widetilde{T}}\|\nabla u^m\|^p_pdt=\int_0^{\widetilde{T}}\int_\Omega|u^m|^q\ln|u^m|dxdt.
\end{equation}
Letting $m\to \infty$ in \eqref{4.11}, and noticing that 
\[\int_0^{\widetilde{T}}(|x|^{-s}u^m_t,\omega)dt=\int_0^{\widetilde{T}}(|x|^{-\frac s2}u^m_t,|x|^{-\frac s2}\omega)dt\to \int_0^{\widetilde{T}}(|x|^{-\frac s2}u_t,|x|^{-\frac s2}\omega)dt=\int_0^{\widetilde{T}}(|x|^{-s}u_t,\omega)dt,\]
we obtain 
\begin{equation}
\label{4.13}
\int_0^{\widetilde{T}}(|x|^{-s}u_t,\omega)dt+\int_0^{\widetilde{T}}(\xi,\nabla\omega)dt=\int_0^{\widetilde{T}}(|u|^{q-2}u\ln|u|,\omega)dt.
\end{equation}
 Choosing $\omega=u^m$ in \eqref{4.13}, and letting $m\to \infty$, one has
 \begin{equation}
\label{4.14}
\frac12\int_\Omega|x|^{-s}(u(x,\widetilde{T}))^2dx-\frac12\int_\Omega|x|^{-s}(u_0(x))^2dx+\int_0^{\widetilde{T}}(\xi,\nabla u)dt=\int_0^{\widetilde{T}}(|u|^{q-2}u\ln|u|,u)dt.
\end{equation}

 We need to prove that $\xi=|\nabla u|^{p-2}\nabla u$. Using Hardy-Sobolev inequality(i.e. Lemma \ref{lem2.5}), we have
 \begin{equation}
\label{4.15}
\Big\||x|^{-\frac s2}u^m\Big\|_2^2=\int_\Omega |x|^{-s}(u^m)^{2}dx\le C\Big(\int_\Omega |\nabla u^m|^{\frac{2N}{N+2-s}}dx\Big)^{\frac{N+2-s}{N}}.
\end{equation}
Taking that $\frac{2N}{N+2-s}\le p$ and $p\ge 2$ in mind, it follows from inequality  $s\le s^\alpha+1$ with $s\ge 0,~\alpha\ge 1$ and \eqref{4.4} that  
\begin{equation}
\label{4.16}
C\Big(\int_\Omega |\nabla u^m|^{\frac{2N}{N+2-s}}dx\Big)^{\frac{N+2-s}{N}}\le \widetilde{C}\|\nabla u^m\|_p^2\le \widetilde{C}\|\nabla u^m\|_p^p+\widetilde{C}\le \widetilde{C}\frac{dpq}{q-p}+\widetilde{C}
\end{equation}
with 
\begin{equation}
\label{20h}
\widetilde{C}=\begin{cases}
C&\quad\text{if } \frac{2N}{N+2-s}= p,\\
C|\Omega|^{\frac{N+2-s}{N}-\frac{2}{p}}&\quad\text{if }\frac{2N}{N+2-s}< p.
\end{cases}
\end{equation}
Combining \eqref{4.15} with \eqref{4.16} yields that there exist a subsequence of $\{|x|^{-\frac s2}u^m(x,\widetilde{T})\}$ (which we still denote by $\{|x|^{-\frac s2}u^m(x,\widetilde{T})\}$) and a function $v\in L^2(\Omega)$
such that $\{|x|^{-\frac s2}u^m(x,\widetilde{T})\}\rightharpoonup v$ weakly in $L^2(\Omega)$. Then for any $\psi(x)\in C_0^\infty (\Omega)$ and $\varphi(t)\in C^1([0, \widetilde{T}])$, one has
\begin{equation*}
\begin{split}
\int_0^{\widetilde{T}}\int_\Omega|x|^{-\frac s2}u_t^m\psi(x)\varphi(t)dxdt&=\int_\Omega\Big[|x|^{-\frac s2}u^m(x,\widetilde{T})\varphi(\widetilde{T})-|x|^{-\frac s2}u^m(x,0)\varphi(0)\Big]\psi(x)dx\\
&\quad-\int_0^{\widetilde{T}}\int_\Omega|x|^{-\frac s2}u^m\psi(x)\varphi_t(t)dxdt.
\end{split}
\end{equation*}
Letting $m\to \infty$ in the equality above, and taking \eqref{4.8} \eqref{4.2} in mind, we obtain 
\begin{equation*}
\int_\Omega \Big[v-|x|^{-\frac s2}u(x,\widetilde{T})\Big]\psi(x)\varphi(\widetilde{T})dx-\int_\Omega \Big[|x|^{-\frac s2}u_0(x)-|x|^{-\frac s2}u(x,0)\Big]\psi(x)\varphi(0)dx=0.
\end{equation*}
Setting $\varphi(\widetilde{T})=1$, $\varphi(0)=0$, and by the density of $C_0^\infty (\Omega)$ in $L^2(\Omega)$, we have $v=|x|^{-\frac s2}u(x,\widetilde{T})$. It follows from the weakly lower semi-continuity of the $L^2(\Omega)$ norm that 
\begin{equation}
\label{4.17}
\int_\Omega|x|^{-s}(u(x,\widetilde{T})^2dx\le \liminf_{m\to \infty}\int_\Omega|x|^{-s}(u^m(x,\widetilde{T})^2dx.
\end{equation}
Next, we claim that
\begin{equation}
\label{4.18}
\lim_{m\to \infty}\int_0^{\widetilde{T}}\int_\Omega|u^m|^q\ln|u^m|dxdt=\int_0^{\widetilde{T}}\int_\Omega|u|^q\ln|u|dxdt.
\end{equation}
Using \eqref{4.10}, \eqref{4.7} and the last one in \eqref{4.8}, one has that 
\begin{equation*}
\begin{split}
&\Big|\int_0^{\widetilde{T}}\int_\Omega\Big(|u^m|^q\ln|u^m|-|u|^q\ln|u|\Big)dxdt\Big|
\\&\quad=\Big|\int_0^{\widetilde{T}}\int_\Omega(u^m-u)u^m|u^m|^{q-2}\ln|u^m|dxdt\Big|\\
&\quad\quad+
\Big|\int_0^{\widetilde{T}}\int_\Omega u\Big(|u^m|^{q-2}u^m\ln|u^m|-|u|^{q-2}u\ln|u|\Big)dxdt\Big|\\
&\quad\le \max_{0\le t\le\widetilde{T}}\|u^m-u\|_q\max_{0\le t\le \widetilde{T}}\Big\|u^m|^{q-2}u^m\ln|u^m|\Big\|_{\frac{q}{q-1}}\\
&\quad\quad+\Big|\int_0^{\widetilde{T}}\int_\Omega u\Big(|u^m|^{q-2}u^m\ln|u^m|-|u|^{q-2}u\ln|u|\Big)dxdt\Big|\to 0\quad\text{as }m\to \infty,
\end{split}
\end{equation*}
which yields \eqref{4.18}.

Set 
\[X_m=\int_0^{\widetilde{T}}\int_\Omega(|\nabla u^m|^{p-2}\nabla u^m-|\nabla u|^{p-2}\nabla u)(\nabla u^m-\nabla u )dxdt,\] obviously, it follows from \eqref{4.12}, \eqref{4.18}, \eqref{4.17}, \eqref{4.8} and \eqref{4.14}  that 
\begin{equation*}
\begin{split}
0&\le \limsup_{m\to \infty}X_m=\limsup_{m\to \infty}\int_0^{\widetilde{T}}\int_\Omega\Big[|\nabla u^m|^p-|\nabla u^m|^{p-2}\nabla u^m\nabla u-|\nabla u|^{p-2}\nabla u(\nabla u^m-\nabla u)\Big]dxdt\\
&\le \int_0^{\widetilde{T}}\int_\Omega|u|^q\ln|u|dxdt-\frac12\int_\Omega|x|^{-s}(u(x,T))^2dx+\frac12\int_\Omega|x|^{-s}(u(x,0))^2dx\\
&\quad- \int_0^{\widetilde{T}}\int_\Omega\xi\nabla u dxdt-\int_0^{\widetilde{T}}\int_\Omega|\nabla u|^{p-2}\nabla u(\nabla u-\nabla u)dxdt=0,
\end{split}
\end{equation*}
which implies $\lim_{m\to \infty}X_m=0.$ Since for $p\ge 2$,
\[\int_0^{\widetilde{T}}\int_\Omega|\nabla u^m-\nabla u|^pdxdt\le 2^{p-2}X_m,\]
one has that $\nabla u^m\to \nabla u$ strongly in $(L^p(0,\widetilde{T};L^p(\Omega)))^N,$ which implies $\nabla u^m\to \nabla u$ a.e. in $\Omega\times (0,\widetilde{T})$. Consequently  $|\nabla u^m|^{p-2}\nabla u^m\to |\nabla u|^{p-2}\nabla u$ a.e. in $\Omega\times (0,\widetilde{T})$, which together with the third one in \eqref{4.8} yields $\xi=|\nabla u|^{p-2}\nabla u$. Therefore, it follows from \eqref{4.13}, for any $\omega\in C^1([0,\widetilde{T}],W_0^{1,p}(\Omega))$, that 
\begin{equation*}
\int_0^{\widetilde{T}}(|x|^{-s}u_t,\omega)dt+\int_0^{\widetilde{T}}(|\nabla u|^{p-2}\nabla u,\nabla\omega)dt=\int_0^{\widetilde{T}}(|u|^{q-2}u\ln|u|,\omega)dt.
\end{equation*}
By the arbitrariness of $\widetilde{T} > 0$, we know that 
\begin{equation*}
(|x|^{-s}u_t,\phi)+(|\nabla u|^{p-2}\nabla u,\nabla\phi)=(|u|^{q-2}u\ln|u|,\phi)\quad \text{for any }\phi\in W_0^{1,p}(\Omega), \text{ a.e. }t>0.
\end{equation*}

In order to prove \eqref{19add2}, we first assume that $u(x, t)$ is smooth enough such that
$u_t \in L^2(0, \widetilde{T}; W_0^{1,p}(\Omega))$. Let us choose
$\phi= u_t $ as a test function and integrate \eqref{19add1} over $[0, t]$, obviously, \eqref{19add2} is true. Making use of  the density of $L^2(0, \widetilde{T}; W_0^{1,p}(\Omega))$ in $L^2(0, \widetilde{T}; L^2(\Omega))$, \eqref{19add2} also holds  for weak solutions of \eqref{1.1}. Therefore, $u$ is a global weak solution of problem \eqref{1.1}.

\textsc{Step 2. Decay rate}  Obviously, $u\in\mathcal{W}$ for $t\in [0,\infty)$ from \textsc{Step 1}, further, $I(u)\ge 0.$ It follows from \eqref{19add2} and \eqref{19add5} that 
\begin{equation*}
J(u_0)\ge J(u(x,t))\ge \frac{q-p}{pq}\|\nabla u\|_p^p,
\end{equation*}
which yields, together with the embedding $W_0^{1,p}(\Omega)\hookrightarrow  L^{p+\alpha}(\Omega)$ with $\alpha$ defined in \eqref{20e},
\[\|u\|_{q+\alpha}\le B_\alpha\|\nabla u\|_p\le B_\alpha\Big(\frac{pq}{q-p}J(u_0)\Big)^{\frac1p}.\]
Consequently, 
\begin{equation}
\label{21a}
\|u\|_{q+\alpha}^{q+\alpha}=\|u\|_{q+\alpha}^{p}\|u\|_{q+\alpha}^{q+\alpha-p}\le B_\alpha^{q+\alpha}\|\nabla u\|_p^p \Big(\frac{pq}{q-p}J(u_0)\Big)^{\frac{q+\alpha-p}{p}}.
\end{equation}
Combining  \eqref{20add1} and \eqref{21a}, and recalling \eqref{20f}, one has
\begin{equation*}
\begin{split}
\frac{d}{dt}\Big\||x|^{-\frac s2}u\Big\|_2^2&=-2I(u)\le -2\Big(\|\nabla u\|_p^p-\frac{1}{\alpha}\|u\|_{q+\alpha}^{q+\alpha}\Big)\\
&\quad \le -2\|\nabla u\|_p^p\Big[1-\frac{B_\alpha^{q+\alpha}}{\alpha}\Big(\frac{pq}{q-p}J(u_0)\Big)^{\frac{q+\alpha-p}{p}}\Big]\\
&\quad =-2\|\nabla u\|_p^p\Big[1-\Big(\frac{pq}{q-p}\frac{1}{(r(\alpha))^p}J(u_0)\Big)^{\frac{q+\alpha-p}{p}}\Big].
\end{split}
\end{equation*}
Replacing $u^m$ by $u$ in \eqref{4.15} and \eqref{4.16}, we get 
\[\Big\||x|^{-\frac s2}u\Big\|_2^2\le \widetilde{C}\|\nabla u\|_p^2.\]
Thus,
\begin{equation*}
\begin{split}
\frac{d}{dt}\Big\||x|^{-\frac s2}u\Big\|_2^2\le -\frac{2}{\widetilde{C}^{\frac p2}}\Big(\Big\||x|^{-\frac s2}u\Big\|_2^2\Big)^{\frac p2}\Big[1-\Big(\frac{2q}{q-2}\frac{1}{(r(\alpha))^2}J(u_0)\Big)^{\frac{q+\alpha-2}{2}}\Big],
\end{split}
\end{equation*}
which implies
\begin{equation*}
\Big\||x|^{-\frac s2}u\Big\|_2^2\le
\begin{cases}
 \Big\||x|^{-\frac s2}u_0\Big\|_2^2e^{-\frac{2}{\widetilde{C}}\Big[1-\Big(\frac{2q}{q-2}\frac{1}{(r(\alpha))^2}J(u_0)\Big)^{\frac{q+\alpha-2}{2}}\Big]t}&\text{for }p=2,\\
\Big\{\Big(\frac p2-1\Big)\frac{2}{\widetilde{C}^{\frac p2}}\Big[1-\Big(\frac{pq}{q-p}\frac{1}{(r(\alpha))^p}J(u_0)\Big)^{\frac{q+\alpha-p}{p}}\Big]t+\Big(\Big\||x|^{-\frac s2}u_0\Big\|_2^2\Big)^{1-\frac p2}\Big\}^{\frac{2}{2-p}}&\text{for }p>2.
 \end{cases}
\end{equation*}
   \end{proof1}
 
 \begin{proof2} 
 For this proof, we need to use the technique introduced by Philippin and Proytcheva in \cite{PP2006}, and further developed by Han et al. \cite{HGSL2018}. Let $K(t):=-J(u(x,t))$, then $L(0)>0$, $K(0)>0.$ It follows from \eqref{19add2} that 
\[K'(t)=-\frac{d}{dt}J(u(x,t))=\Big\||x|^{-\frac s2}u_t\Big\|_2^2\ge 0,\]
which yields $K(t)\ge K(0)>0$ for $t\in [0,T).$ Recalling Lemma \ref{20lemadd1} and \eqref{19add5}, one has
\begin{equation}
\label{5.2}
L'(t)=-I(u(x,t))=\frac {q-p}{p}\|\nabla u\|_p^p+\frac {1}{q}\|u\|_q^q-qJ(u(x,t))\ge qK(t).
\end{equation}
Combining Schwarz's inequality and \eqref{5.2} gets
\[L(t)K'(t)=\frac 12\Big\||x|^{-\frac s2}u\Big\|_2^2\Big\||x|^{-\frac s2}u_t\Big\|_2^2\ge \frac 12(L'(t))^2\ge \frac q2L'(t)K(t),\]
which implies
\[\Big(K(t)L^{-\frac q2}(t)\Big)'=L^{-\frac {q+2}{2}}(t)\Big[K'(t)L(t)-\frac q2 K(t)L'(t)\Big]\ge 0.\]
Therefore, together with \eqref{5.2}, we have 
\begin{equation}
\label{5.3}
0<\kappa:=K(0)L^{-\frac q2}(0)\le K(t)L^{-\frac q2}(t)\le \frac 1qL'(t)L^{-\frac q2}(t)=\frac{2}{(2-q)q}\Big(L^{\frac {2-q}{2}}(t)\Big)'.
\end{equation}
Integrating the inequality above over $[0,t]$ for $t\in (0,T)$, and taking $q>p\ge 2$ in mind, one has
\[0\le L^{\frac {2-q}{2}}(t)\le L^{\frac {2-q}{2}}(0)-\frac{(q-2)q}{2}\kappa t\quad \text{for }t\in(0,T).\]
Obviously, the inequality above cannot hold for all $t>0$, i.e. $T<\infty$. Moreover, 
\[T\le \frac{2}{(q-2)q\kappa}L^{\frac {2-q}{2}}(0)= \frac{2L(0)}{(2-q)qJ(u_0)}.\]
 \end{proof2}
 
 \begin{proofaddthm2}
In order to prove this theorem, let us borrow some ideas from \cite{SLW2018}. From Lemma \ref{20lemadd2}(2)(3), let us conclude that  there exists a $t_0\in [0, T )$  such that $J(u(x,t))< M$ and $I(u(x,t))<0$ for all $t\in [t_0, T )$ provided that $J(u_0) \le M$ and $I(u_0) < 0$.  Therefore, it follows from Lemma \ref{1lem20}(2)  and \eqref{20g} that 
 \begin{equation}
\label{5.4}
\|\nabla u\|_p^p\ge r_*^p=\frac{pq}{q-p}M \quad\text{for }t\in [t_0, T ).
\end{equation}

 Next, we prove $T<\infty$.
 For any $T^*\in(0,T)$, $\gamma>0$ and $\sigma>0$, define an auxiliary function 
 \[M(t)=\int_{t_0}^tL(\tau)d\tau+(T-t)L(t_0)+\frac{\gamma}{2}(t+\sigma)^2\quad\text{for }t\in[t_0,T^*].\]
 By a direct computation, one has
 \begin{equation*}
\begin{split}
M'(t)=L(t)-L(t_0)+\gamma(t+\sigma)=\int_{t_0}^t\int_\Omega |x|^{-s}uu_\tau dxd\tau+\gamma(t+\sigma)\quad\text{for }t\in[t_0,T^*].
\end{split}
\end{equation*}
Further, recall \eqref{20add1}, \eqref{19add5} and \eqref{19add2}, then
 \begin{equation*}
\begin{split}
M''(t)&=\int_\Omega |x|^{-s}uu_t dx+\gamma=-I(u(x,t))+\gamma\\
&=\frac {q-p}{p}\|\nabla u\|_p^p+\frac {1}{q}\|u\|_q^q-qJ(u(x,t))+\gamma\\
&
\ge \frac {q-p}{p}\|\nabla u\|_p^p -q\Big[J(u(x,t_0))-\int_{t_0}^t\Big\||x|^{-\frac s2}u_\tau\Big\|_2^2d\tau\Big]+\gamma\quad \text{for }t\in[t_0,T^*].
\end{split}
\end{equation*}
Applying Cauchy-Schwarz inequality, one has 
 \begin{equation*}
 \begin{split}
\xi(t):&=\Big[\int_{t_0}^t\Big\||x|^{-\frac s2}u\Big\|_2^2d\tau+\gamma(t+\sigma)^2\Big]\Big[\int_{t_0}^t\Big\||x|^{-\frac s2}u_\tau\Big\|_2^2 d\tau+\gamma\Big]\\
&\quad-\Big[\int_{t_0}^t\int_\Omega |x|^{-s}uu_\tau dxd\tau+\gamma(t+\sigma)\Big]^2\ge 0\quad\text{for }t\in[{t_0},T^*].
\end{split}
\end{equation*}
 Therefore, recalling \eqref{5.4}, 
 \begin{equation}
\label{5.5}
\begin{split}
&M(t)M''(t)-\frac{q}{2}(M'(t))^2\\
&\quad\ge M(t)M''(t)+\frac{q}{2}\Big[\xi(t)-\Big(2M(t)-2(T-t)L(t_0)\Big)\Big(\int_{t_0}^t\Big\||x|^{-\frac s2}u_\tau\Big\|_2^2d\tau+\gamma\Big)\Big]\\
&\quad \ge M(t)\Big[M''(t)-q\Big(\int_{t_0}^t\Big\||x|^{-\frac s2}u_\tau\Big\|_2^2d\tau+\gamma\Big)\Big]\\
&\quad\ge M(t)\Big[\frac {q-p}{p}\|\nabla u\|_p^p-qJ(u(x,t_0))-(q-1)\gamma\Big]\\
&\quad \ge M(t)\Big[qM-qJ(u(x,t_0))-(q-1)\gamma\Big]\ge 0
\end{split}
\end{equation}
for $t\in[{t_0},T^*]$ and $\gamma\in \Big(0,\frac{q(M-J(u(x,t_0)))}{q-1}\Big]$. It follows from Lemma \ref{lem2.4} that 
\[0<T^*-t_0\le \frac{2M(t_0)}{(q-2)M'(t_0)}=\frac{2(T-t_0)L(t_0)}{(q-2)\gamma(t_0+\sigma)}+\frac{t_0+\sigma}{q-2}.\]
Since the arbitrariness of $T^*<T$,  for any $\gamma\in \Big(0,\frac{q(M-J(u(x,t_0)))}{q-1}\Big]$ and $\sigma>0$, one has
\begin{equation}
\label{5.6}
T-t_0\le \frac{2(T-t_0)L(t_0)}{(q-2)\gamma(t_0+\sigma)}+\frac{t_0+\sigma}{q-2}.
\end{equation}

Fix now $\gamma_0\in\Big(0,\frac{q(M-J(u(x,t_0)))}{q-1}\Big]$, then for any  $\sigma\in\Big (\frac{2L(t_0)}{(q-2)\gamma_0}-t_0,+\infty\Big)$, $0<\frac{2L(t_0)}{(q-2)\gamma_0(t_0+\sigma_0)}<1$ holds, which implies together with \eqref{5.6}
\[T\le\frac{\gamma_0(t_0+\sigma)^2}{(q-2)\gamma_0(t_0+\sigma)-2L(t_0)}+t_0.\]
Minimizing the right-hand side of the inequality above for $\sigma\in\Big (\frac{2L(t_0)}{(q-2)\gamma_0}-t_0,+\infty\Big)$, one gets 
\begin{equation}
\label{5.7}
T\le \frac{8L(t_0)}{(q-2)^2\gamma_0}+t_0\quad \text{for }\gamma_0\in\Big(0,\frac{q(M-J(u(x,t_0)))}{q-1}\Big].
\end{equation}
Minimizing the right-hand side of the inequality above for $\gamma_0\in\Big(0,\frac{q(M-J(u(x,t_0)))}{q-1}\Big]$, we obtain 
\[T\le \frac{8(q-1)L(t_0)}{(q-2)^2q(M-J(u(x,t_0)))}+t_0.\]
For $J(u_0)<M$, taking $t_0=0$, then 
\[T\le \frac{8(q-1)L(0)}{(q-2)^2q(M-J(u_0))}.\]
  \end{proofaddthm2}

 \begin{proof3}
 This proof follows the part ideas in \cite{H2018}. Firstly, we prove that $u$ blows up in finite time.  Suppose on the contrary that $u$ is global, i.e. $T=+\infty$.   Then, recalling \eqref{5.1}, for all $t\in [0,\infty)$, Schwarz's inequality and \eqref{19add2} imply 
 \begin{equation}
\label{6.1}
\begin{split}
(2L(t))^{\frac12}&=\Big\||x|^{-\frac s2}u\Big\|_2=\Big\|\int_0^t|x|^{-\frac s2}u_\tau d\tau+|x|^{-\frac s2}u_0\Big\|_2
\le\int_0^t\Big\||x|^{-\frac s2}u_\tau \Big\|_2d\tau+(2L(0))^{\frac12} \\
&\le t^{\frac12}\Big(\int_0^t\Big\||x|^{-\frac s2}u_\tau \Big\|_2^2d\tau\Big)^{\frac12}+(2L(0))^{\frac12}
\le t^{\frac12}(J(u_0)-J(u(x,t)))^{\frac12}+(2L(0))^{\frac12}\\
&\le t^{\frac12}J^{\frac12}(u_0)+(2L(0))^{\frac12},
\end{split}
\end{equation}
 where we apply $0\le J(u(x,t))\le J(u_0)$ if $u$ is a global solution. Here, we prove that $0\le J(u(t))\le J(u_0)$. Otherwise,
there exists $t_*\in [0, \infty)$ such that $J(t_*) < 0$. Then by Theorem \ref{thm2}, we know that $u$ blows up in finite time, which is a contradiction. Replacing $u^m$ by $u$ in \eqref{4.15} and \eqref{4.16}, one has
\begin{equation}
\label{6.2}
2L(t)=\Big\||x|^{-\frac s2}u\Big\|_2^2\le \widetilde{C}\|\nabla u\|_p^p+\widetilde{C}.
\end{equation}
 It follows from  \eqref{5.2} and \eqref{6.2} that 
 \begin{equation}
\label{6.3}
\begin{split}
L'(t)=&-I(u(x,t))=\frac {q-p}{p}\|\nabla u\|_p^p+\frac {1}{q}\|u\|_q^q-qJ(u(x,t))\\
&\ge \frac {q-p}{p}\Big(\frac{2}{\widetilde{C}}L(t)-1\Big)-qJ(u(x,t))\\
&=\frac {q-p}{p}\frac{2}{\widetilde{C}}\Big(L(t)-C_1-C_2J(u(x,t)\Big),
\end{split}
\end{equation}
here $C_1=\frac{\widetilde{C}}{2}$ and  $C_2=\frac {pq}{q-p}\frac{\widetilde{C}}{2}$.
 Set 
\begin{equation}
\label{6.4}
F(t)=L(t) -C_1-C_2J(u(t)),
\end{equation}
then by  \eqref{6.3}
\begin{equation*}
F'(t)\ge L'(t)
\ge \frac {q-p}{p}\frac{2}{\widetilde{C}}F(t).
\end{equation*}
Since $F(0)=L(0)-C_1-C_2J(u_0)>0$ from \eqref{20add4}, we get 
\begin{equation}
\label{6.5}
F(t)\ge F(0)e^{\frac {q-p}{p}\frac{2}{\widetilde{C}}t}>0.
\end{equation}
Therefore, by \eqref{6.4}, one has
\[L(t)\ge F(t)\ge F(0)e^{\frac {q-p}{p}\frac{2}{\widetilde{C}}t},\]
which contradicts \eqref{6.1} for sufficiently large $t$. Thus, $u$ blows up in finite time. Moreover, \eqref{6.3} and \eqref{6.5} imply  that $L(t)$ is strictly increasing for $t\in [0,\infty)$.

Secondly, let us estimate the upper bound of $T$. For any $T^*\in(0,T)$, $\gamma>0$ and $\sigma>0$, define an auxiliary function 
 \[M(t)=\int_0^tL(\tau)d\tau+(T-t)L(0)+\frac{\gamma}{2}(t+\sigma)^2\quad\text{for }t\in[0,T^*].\]
 Similar  to \eqref{5.5},  and noticing that \eqref{6.2} and  $L(t)$ is strictly increasing for $t\in [0,\infty)$, we have
 \begin{equation}
\label{6.6}
\begin{split}
M(t)M''(t)-\frac{q}{2}(M'(t))^2&\ge M(t)\Big[\frac {q-p}{p}\|\nabla u\|_p^p-qJ(u_0)-(q-1)\gamma\Big]\\
&\ge M(t) \Big[\frac {q-p}{p}\Big(\frac{2}{\widetilde{C}}L(t)-1\Big)-qJ(u_0)-(q-1)\gamma\Big]\\
&\ge M(t) \Big[\frac {q-p}{p}\Big(\frac{2}{\widetilde{C}}L(0)-1\Big)-qJ(u_0)-(q-1)\gamma\Big]\\
&=M(t) \Big[\frac {q-p}{p}\frac{2}{\widetilde{C}}F(0)-(q-1)\gamma\Big]\ge 0
\end{split}
\end{equation}
for $t\in[0,T^*]$ and $\gamma\in \Big(0,\frac {q-p}{p}\frac{2}{\widetilde{C}}\frac{1}{q-1}F(0)\Big]$.  In the remaining proof,  using the same method in the proof of Theorem \ref{addthm2}, one has 
\[T\le \frac{4(q-1)p\widetilde{C}L(0)}{(q-2)^2(q-p)F(0)}.\]
 \end{proof3}

  \begin{proof4}
 From \eqref{6.3} and \eqref{6.5}, we directly get $I(u(x,t))<0$ for $t\in[0,T),$  which implies
 \begin{equation}
\label{6.7}
\|\nabla u\|_p^p<\int_\Omega |u|^q\ln|u|dx<\frac{1}{\alpha}\|u\|_{q+\alpha}^{q+\alpha},
\end{equation}
here $\alpha$ is defined in \eqref{20e}. 
Applying interpolation inequality, $W_0^{1,p}(\Omega)\hookrightarrow L^{\frac{Np}{N-p}}(\Omega)$ and $I(u(t))<0$ for $t\in[0,T)$, it follows that \begin{equation}
\label{6.8}
\begin{split}
\|u\|_{q+\alpha}^{q+\alpha}&\le \|u\|_{\frac{Np}{N-p}}^{\theta(q+\alpha)}\|u\|_{2}^{(1-\theta)(q+\alpha)}\le C_*^{\theta(q+\alpha)}\|\nabla u\|_{p}^{\theta(q+\alpha)}\|u\|_{2}^{(1-\theta)(q+\alpha)}\\
&<C_*^{\theta(q+\alpha)}\alpha^{-\frac{\theta(q+\alpha)}{p}}\Big(\|u\|_{q+\alpha}^{q+\alpha}\Big)^{\frac{\theta(q+\alpha)}{p}}\Big(\|u\|_{2}^{2}\Big)^{\frac{(1-\theta)(q+\alpha)}{2}}\\
&=C_*^{\theta(q+\alpha)}\alpha^{-\frac{\theta(q+\alpha)}{p}}[\text{diam}(\Omega)]^{\frac{s(1-\theta)(q+\alpha)}{2}}\Big(\|u\|_{q+\alpha}^{q+\alpha}\Big)^{\frac{\theta(q+\alpha)}{p}}\Big(2L(t)\Big)^{\frac{(1-\theta)(q+\alpha)}{2}}.
\end{split}
\end{equation}
 By recalling $\theta=\Big(\frac{1}{2}-\frac{1}{q+\alpha}\Big)\Big(\frac{1}{2}-\frac{N-p}{Np}\Big)^{-1}\in(0,1)$ and $q+\alpha<p\big(1+\frac 2N\big)$, one has $1-{\frac{\theta(q+\alpha)}{p}}>0$, and $\kappa=[\frac{(1-\theta)(q+\alpha)}{2}]/[1-\frac{\theta(q+\alpha)}{p}]>1.$ Therefore,
 \begin{equation}
\label{6.9}
\begin{split}
\frac{d}{dt}L(t)&=-I(u(t))\le \int_\Omega |u|^q\ln|u|dx<\frac{1}{\alpha}\|u\|_{q+\alpha}^{q+\alpha}\\
&<\frac 1\alpha\Big[C_*^{\theta(q+\alpha)}\alpha^{-\frac{\theta(q+\alpha)}{p}}[\text{diam}(\Omega)]^{\frac{s(1-\theta)(q+\alpha)}{2}}\Big]^{\frac{p}{p-\theta(q+\alpha)}}2^\kappa L^\kappa(t).
\end{split}
\end{equation}
 And $L(t)>0$ due to $I(u(t))<0$ for $t\in[0,T)$. Further, \eqref{6.9} yields 
 \[\frac{1}{1-\kappa}(L^{1-\kappa}(t)-L^{1-\kappa}(0))\le \frac 1\alpha\Big[C_*^{\theta(q+\alpha)}\alpha^{-\frac{\theta(q+\alpha)}{p}}[\text{diam}(\Omega)]^{\frac{s(1-\theta)(q+\alpha)}{2}}\Big]^{\frac{p}{p-\theta(q+\alpha)}}2^\kappa t.\]
 By Theorem $\ref{thm3}$, we get $\lim_{t\to T^-}L(t)dt=+\infty.$  As a result, by letting $t\to T$, we obtain 
  \begin{equation*}
\label{20add6}
T\ge \frac{L^{1-\kappa}(0)}{\alpha^{-1}\Big[C_*^{\theta(q+\alpha)}\alpha^{-\frac{\theta(q+\alpha)}{p}}[\text{diam}(\Omega)]^{\frac{s(1-\theta)(q+\alpha)}{2}}\Big]^{\frac{p}{p-\theta(q+\alpha)}}2^\kappa(\kappa-1)}.
\end{equation*}
  \end{proof4}

\subsection*{Acknowledgements}
The  first author would like  to express her sincere gratitude to Professor Wenjie Gao for his support and constant encouragement. This work is supported by the National Natural Science Foundation of China(No.  11926316, 11531010, 12071391).

\subsection*{Competing interests}
The authors declare that they have no competing interests.

\end{document}